\documentclass[11pt]{amsart}

\usepackage{amssymb,amsmath}
\usepackage{enumerate}

\newtheorem{theorem}{Theorem}

\newtheorem{corollary}{Corollary}

\theoremstyle{remark}
\newtheorem{example}{Example}
\newtheorem{remark}{Remark}

\usepackage{hyperref}

\newcommand{\E}{\mathcal{E}}
\newcommand{\ES}{\mathcal{E}}
\newcommand{\Ex}{\mathbb{E}}

\newcommand{\life}{t_+}
\newcommand{\nlife}{t_{-}}

\newcommand{\BS}{\mathcal{X}}
\newcommand{\Ker}{\mathcal{N}}

\renewcommand{\xi}{X}

\newcommand{\sQ}{\mathfrak{Q}}
\renewcommand{\vartheta}{\rho}
\newcommand{\Ha}{H}

\renewcommand{\psi}{h}

\begin{document}

\title{Positive semigroups and perturbations of boundary conditions}
\thanks{
This work was partially supported  by the Polish NCN grant  2017/27/B/ST1/00100 and by the grant 346300 for IMPAN from the Simons Foundation and the matching 2015-2019 Polish MNiSW fund.
}

\author{Piotr Gwi\.zd\.z }
\address{Institute of Mathematics, University of Silesia, Bankowa 14, 40-007
Katowice, Poland}
\email{piotr.gwizdz@gmail.com}

\author{Marta Tyran-Kami\'nska}
\address{Institute of Mathematics, Polish Academy of Sciences, Bankowa 14, 40-007
Katowice, Poland}
\email{mtyran@us.edu.pl}


\begin{abstract}
We present a generation theorem for positive semigroups on an $L^1$ space. It provides  sufficient conditions for the existence of positive and integrable solutions of initial-boundary value problems. An application to a two-phase cell cycle model is given.
\end{abstract}

\maketitle

\section{Introduction}

We study well-posedness of linear evolution equations on $L^1$  of the form
\begin{equation}\label{e:eq}
u'(t) =Au(t),\quad \Psi_0 u(t)=\Psi u(t),\quad t> 0,\quad u(0)=f,
\end{equation}
where $\Psi_0,\Psi$ are positive and possibly unbounded linear operators on $L^1$, the linear operator $A$ is such that equation \eqref{e:eq} with $\Psi=0$ generates a \emph{positive semigroup} on $L^1$, i.e., a~$C_0$-semigroup of positive operators on $L^1$.  We present sufficient conditions for the operators $A, \Psi_0,$ and $\Psi$ under which there is a unique positive semigroup on $L^1$ providing solutions of the initial-boundary value problem~\eqref{e:eq}. For a general theory of positive semigroups and their applications we refer the reader to~\cite{nagel86_0,banasiakarlotti06,bobrowski16,kramarrhandi17,rudnickityran17}. An overview of  different approaches used in studying initial-boundary value problems is presented in \cite{bobrowski1015}.

Our result is an extension of Greiner's~\cite{greiner} by considering unbounded $\Psi$ and positive semigroups.  Unbounded perturbations of the boundary conditions of a generator were studied recently in \cite{adler2014perturbation} and \cite{adler2017perturbation} by using extrapolated spaces and
various admissibility conditions.
In the proof of our perturbation theorem  we apply a result  about positive perturbations of resolvent positive operators \cite{arendt87} with non-dense domain in  $AL$-spaces in the form given in
\cite[Theorem 1.4]{thieme96}. It is an extension of the well known perturbation result due to Desch \cite{desch} and by Voigt \cite{voigt89}. For positive perturbations of positive semigroups in the case when the space is not an $AL$-space we refer to~\cite{arendtrhandi91,voigt2018AM}.
We also present a result about stationary  solutions of \eqref{e:eq}. We illustrate our general results with an age-size-dependent cell
cycle model generalizing the discrete time model of \cite{hannsgen85a,al-mcm-jt-92,tyrcha01}. This model can be described as a piecewise deterministic Markov process (see Section \ref{s:fr} and \cite{rudnickityran17}). Our approach can also be used in transport equations \cite{kramarsikolya05,banasiakfalkiewicz15}.

\section{General Results}
\label{s:pre}

Let $(E,\ES,m)$ and $(E_{\partial},\ES_{\partial},m_{\partial})$ be two $\sigma$-finite measure spaces. Denote by $L^1=L^1(E,\ES,m)$ and $L^1_{\partial}=L^1(E_{\partial},\ES_{\partial},m_{\partial})$ the corresponding spaces of integrable functions.
Let $\mathcal{D}$ be a linear subspace of $ L^1$.
We assume that $A\colon \mathcal{D}\to L^1$ and $\Psi_0,\Psi\colon \mathcal{D}\to L^1_{\partial}$ are linear operators satisfying the following conditions:
\begin{enumerate}[(i)]
\item\label{S1} for each  $\lambda>0$, the operator $\Psi_0\colon \mathcal{D}\to L^1_{\partial}$ restricted to the nullspace $\Ker(\lambda I-A)=\{f\in \mathcal{D}:\lambda f-Af=0\}$ of the operator $(\lambda I-A,\mathcal{D})$ has a positive right inverse, i.e., there exists a positive operator $\Psi(\lambda)\colon L^1_{\partial} \to \Ker(\lambda I-A)$ such that $\Psi_0\Psi(\lambda)f_{\partial}=f_{\partial}$ for $f_{\partial}\in L^1_{\partial}$;
    \item\label{S3} the operator $\Psi\colon \mathcal{D} \to L^1_{\partial}$ is positive and there exists $\omega\in \mathbb{R}$ such that the operator $I_{\partial}-\Psi \Psi(\lambda)\colon L^1_{\partial}\to L^1_{\partial}$ is invertible with positive inverse for all $\lambda>\omega$, where $I_{\partial}$ is the identity operator on $L^1_{\partial}$;

\item\label{S2} the operator $A_0\subseteq A$ with $\mathcal{D}(A_0)=\{f\in \mathcal{D}:\Psi_0f=0\}$ is the generator of a positive semigroup on $L^1$;
        \item\label{S4} for each nonnegative  $f\in \mathcal{D}$
        \begin{equation}\label{a:lzero}
        \int_E Af(x)\,m(dx)-\int_{E_\partial}\Psi_0f(x)\,m_{\partial}(dx)\le 0.
        \end{equation}
   \end{enumerate}

\begin{theorem}\label{l:Apsi}
Assume conditions \eqref{S1}--\eqref{S4}.
Then the operator  $(A_{\Psi},\mathcal{D}(A_\Psi))$ defined by
\begin{equation}
\label{d:dombou}
A_{\Psi}f=Af, \quad f\in \mathcal{D}(A_{\Psi})=\{f\in \mathcal{D}: \Psi_0(f)=\Psi(f)\},
\end{equation}
is
the generator of a positive  semigroup on $L^1$.
Moreover, the resolvent operator of $A_{\Psi}$ at $\lambda>\omega$ is given by
\begin{equation}\label{eq:resbou}
R(\lambda,A_{\Psi})f=(I+\Psi(\lambda)(I_{\partial}-\Psi \Psi(\lambda))^{-1}\Psi)R(\lambda,A_0)f,\quad f\in L^1.
\end{equation}
\end{theorem}

\begin{proof}
The space $\BS=L^1\times L^1_{\partial}$ is an $AL$-space  with norm
 \[
\|(f,f_{\partial})\|=\int_E |f(x)|\,m(dx)+
\int_{E_{\partial}}|f_{\partial}(x)|\,m_{\partial}(dx),\quad (f,f_{\partial})\in L^1\times L^1_{\partial}.
\]
We define operators $\mathcal{A},\mathcal{B}\colon \mathcal{D}(\mathcal{A}) \to L^1\times L^1_{\partial}$ with $\mathcal{D}(\mathcal{A})=\mathcal{D} \times \{0\}$  by (see e.g. \cite{rudnickityran17})
\[
\mathcal{A}(f,0)=(Af,-\Psi_{0}f)\quad\text{and}\quad  \mathcal{B}(f,0)=(0,\Psi f)\quad \text{for }
f\in \mathcal{D}.
\]
We have $\mathcal{D}(A_0)\times \{0\}\subset \mathcal{D}(\mathcal{A})\subset L^1\times \{0\}$ and $\mathcal{D}(A_0)$ is dense in $L^1$. Hence,  $\overline{\mathcal{D}(\mathcal{A})}=L^1\times \{0\}$.
For every  $\lambda>0$ the resolvent of the operator $\mathcal{A}$ at $\lambda>0$ is given by
\begin{equation}\label{eq:RA}
R(\lambda,  \mathcal{A})(f,f_{\partial})=(R(\lambda,A_0)f+\Psi(\lambda) f_{\partial},0),\quad (f,f_{\partial})\in L^1\times L^1_{\partial}.
\end{equation}
Thus $(\mathcal{A},\mathcal{D}(\mathcal{A}))$ is resolvent positive, i.e., its resolvent operator $R(\lambda,\mathcal{A})$ is positive for all sufficiently large $\lambda>0$.
We now show that$\|\lambda R(\lambda,  \mathcal{A})\|\le 1$  for all $\lambda >0$.
Since  the operator $\lambda R(\lambda,  \mathcal{A})$ is positive, it is enough to show that
\begin{equation}\label{e:positive}
\|\lambda R(\lambda,  \mathcal{A})(f,f_{\partial})\|\le \|(f,f_{\partial})\|\quad \text{for nonnegative } (f,f_{\partial})\in L^1\times L^1_{\partial}.
\end{equation}
The operator $R(\lambda,A_0)$ is positive,  $R(\lambda,A_0)f\in \mathcal{D}(A_0)\subseteq \mathcal{D}$ and  $ \Psi_0 R(\lambda,A_0)f=0$ for $f\in L^1$. From this and \eqref{a:lzero} we see that
\[
\int_{E} AR(\lambda,A_0)f(x)\,m(dx)\le \int_{E_{\partial}} \Psi_0 R(\lambda,A_0)f(x)\,m_{\partial}(dx)=0
\]
for all nonnegative $f\in L^1$.
We have  $AR(\lambda,A_0)f=\lambda R(\lambda,A_0)f-f$ for all $f\in L^1$, by \eqref{S2}. Thus, we get
\begin{align*}
\int_{E} \lambda R(\lambda,A_0)f(x)\,m(dx)&=\int_{E} AR(\lambda,A_0)f(x)\,m(dx) +\int_E f(x)\,m(dx)\\
&\le \int_E f(x)\,m(dx),\quad f\in L^1, f\ge 0.
\end{align*}
By assumption \eqref{S1}, $A\Psi(\lambda)f_{\partial}=\lambda \Psi(\lambda)f_{\partial}$ and $\Psi_0 \Psi(\lambda)f_{\partial}=f_{\partial}$ for $f_{\partial}\in L^1_{\partial}$. This together with condition \eqref{a:lzero} implies that
\begin{align*}
\int_{E_{\partial}}\lambda \Psi(\lambda)f_{\partial}(x)\,m_{\partial}(dx)&=\int_{E}A \Psi(\lambda)f_{\partial}(x)\,m(dx)\le \int_{E_{\partial}} \Psi_0 \Psi(\lambda)f_{\partial}(x)\,m_{\partial}(dx)\\
&=\int_{E_{\partial}}f_{\partial}(x)\,m_{\partial}(dx)
\end{align*}
for all nonnegative $f_{\partial}\in L_{\partial}^1$, completing the proof of \eqref{e:positive}.

Let $\mathcal{I}$ be the identity operator on $\BS=L^1\times L^1_{\partial}$. We have $\mathcal{B}R(\lambda,\mathcal{A})(f,f_{\partial})=(0, \Psi R(\lambda,A_0)f+\Psi \Psi(\lambda)f_{\partial})$
for any $(f,f_{\partial})$.
Thus, $\mathcal{I}-\mathcal{B}R(\lambda,\mathcal{A})$ is invertible if and only if $ I_{\partial}-\Psi \Psi(\lambda)$ is  invertible. In that case
\[
(\mathcal{I}-\mathcal{B}R(\lambda,\mathcal{A}))^{-1}(f,f_{\partial})=(f,(I_{\partial}-\Psi\Psi(\lambda))^{-1}(\Psi  R(\lambda,A_0)f+f_{\partial})).
\]
Combining this with \eqref{S3} we conclude that $\mathcal{I}-\mathcal{B}R(\lambda,\mathcal{A})$ is invertible with positive inverse $(\mathcal{I}-\mathcal{B}R(\lambda,\mathcal{A}))^{-1}$ for all  $\lambda>\omega$. Hence, the spectral radius of the positive operator $\mathcal{B}R(\lambda,\mathcal{A})$ is strictly smaller than 1 for some $\lambda>\omega$.
It follows from  \cite[Theorem 1.4]{thieme96} that the part of $(\mathcal{A}+\mathcal{B},\mathcal{D}(\mathcal{A}))$ in $\BS_0=\overline{\mathcal{D}(\mathcal{A})}$ denoted by $((\mathcal{A}+\mathcal{B})_{|},\mathcal{D}((\mathcal{A}+\mathcal{B})_{|}))$ generates a positive  semigroup on $\BS_0$. We have
$
\mathcal{D}((\mathcal{A}+\mathcal{B})_{|})=\mathcal{D}(A_{\Psi})\times \{0\}$ and $ (\mathcal{A}+\mathcal{B})_{|}(f,0)=(A_{\Psi}f,0)$, $f\in \mathcal{D}(A_{\Psi}).$   Consequently, the operator $(A_{\Psi},\mathcal{D}(A_{\Psi}))$ is densely defined and generates a positive semigroup on $L^1$.
Finally, the operator $(\mathcal{A}+\mathcal{B},\mathcal{D}(\mathcal{A}))$ is resolvent positive with resolvent given by $
R(\lambda,\mathcal{A}+\mathcal{B})=R(\lambda,\mathcal{A})(\mathcal{I}-\mathcal{B}R(\lambda,\mathcal{A}))^{-1}
$ for $\lambda>\omega$. Hence, the formula for $R(\lambda,A_{\Psi})$ is also valid.
\end{proof}

\begin{remark} Condition \eqref{S4} ensures that the operator $(A_0,\mathcal{D}(A_0))$
satisfies
\begin{equation}\label{e:subsem0}
\int_E A_0f(x)m(dx)\le 0
\end{equation}
for all nonnegative $f\in \mathcal{D}(A_0)$. If, additionally,
\begin{enumerate}[(v)]
\item\label{cv} $(A_0,\mathcal{D}(A_0))$ is densely defined and resolvent positive,
\end{enumerate}
then $(A_0,\mathcal{D}(A_0))$ is the generator of a \emph{substochastic semigroup}  on $L^1 $, i.e., a positive semigroup of contractions on $L^1$. This is a consequence of the Hille-Yosida theorem, see e.g. \cite[Theorem 4.4]{rudnickityran17}. Thus it is enough  to assume condition  (v) instead of \eqref{S2}.
Observe also that \eqref{S2} and \eqref{S4} imply that $(0,\infty)\subseteq \rho(A_0)$.
\end{remark}

\begin{remark} Note that if  $(A_{\Psi},\mathcal{D}(A_{\Psi}))$ is the generator of a positive semigroup and
\begin{equation}\label{eq:APsiz}
\int_E A_{\Psi}f(x)m(dx)=0 \quad \text{for all nonnegative } f\in \mathcal{D}(A_{\Psi}),
\end{equation}
then $(A_{\Psi},\mathcal{D}(A_{\Psi}))$  generates a \emph{stochastic semigroup}, i.e., a  positive semigroup of operators preserving the $L^1$ norm of nonnegative elements (see e.g. \cite[Section 6.2]{banasiakarlotti06} and \cite[Corollary 4.1]{rudnickityran17}).
\end{remark}

\begin{remark}
If we assume that
\begin{enumerate}[(a)]
\item\label{ca}  $(A,\mathcal{D})$ is closed,
\item\label{cb}  $\Psi_0$ is onto and continuous with respect to the graph norm $\|f\|_A=\|f\|+\|Af\|$,
\end{enumerate}
then $\Psi(\lambda)$ exists for each $\lambda>0$ and is bounded, by
\cite[Lemma 1.2]{greiner}. If $\Psi_0$ is positive, then  $\Psi(\lambda)$ is positive. Thus  condition \eqref{S1} can be replaced by conditions \eqref{ca} and \eqref{cb}.
\end{remark}

\begin{remark}
Greiner~\cite[Theorem 2.1]{greiner} establishes that $(A_{\Psi},\mathcal{D}(A_{\Psi}))$ is the generator of a $C_0$-semigroup for any bounded $\Psi$ provided that conditions  \eqref{ca} and \eqref{cb} hold true, $(A_0,\mathcal{D}(A_0))$ is the generator of a $C_0$-semigroup, and that there exist constants $\gamma>0$ and  $\lambda_0$ such that
\begin{equation}\label{c:grei}
\|\Psi_0 f\|\ge \lambda \gamma\|f\|, \quad f\in \Ker(\lambda I-A),\lambda > \lambda_0.
\end{equation}
This is condition (2.1) of Greiner~\cite[Theorem 2.1]{greiner}. Some extensions of this result are provided in \cite{nickel04} and \cite{rhandi15} for unbounded $\Psi$, as well as in \cite{adler2014perturbation,adler2017perturbation}.
\end{remark}

\begin{remark} Recall that a positive operator on an AL-space defined everywhere is automatically bounded. Thus our assumption \eqref{S1} implies that $\Psi(\lambda)$ is bounded for each $\lambda>0$. Moreover, its norm is determined through its values on the positive cone.  From assumptions \eqref{S1} and \eqref{S4} it follows that $\lambda \|\Psi(\lambda)\|\le 1$ for each $\lambda>0$, as was shown in the proof of Theorem~\ref{l:Apsi}. Thus, for $f=\Psi(\lambda)f_{\partial}$, we get \eqref{c:grei} with $\gamma=1$.
Now suppose, as in \cite{greiner}, that  $\Psi$ is bounded. Then  $\|\Psi\Psi(\lambda)\|\le \|\Psi\|/\lambda$ for all $\lambda>0$. Hence, the operator $I_{\partial}-\Psi\Psi(\lambda)$ is invertible for $\lambda> \|\Psi\|$. Since $I-\Psi(\lambda)\Psi$ is also invertible, we have  $(I-\Psi(\lambda)\Psi)^{-1}=I+\Psi(\lambda)(I_{\partial}-\Psi\Psi(\lambda))^{-1}\Psi$ and, by~\eqref{eq:resbou},
\[
R(\lambda,A_{\Psi})=(I-\Psi(\lambda)\Psi)^{-1}R(\lambda,A_0).
\]
Consequently, if $\Psi$ is bounded and positive, then we get the same result as  in~\cite{greiner}.
\end{remark}

We now look at a simple example where Theorem~\ref{l:Apsi} can be easily applied and it should be compared with \cite[Corollary 25]{adler2014perturbation}.
\begin{example}
Consider the space $L^1=L^1[0,1]$ and the first derivative operator $A= \frac{d}{dx}$ with domain $\mathcal{D}=W^{1,1}[0,1]$. Let $E_{\partial}$ be the one point set $\{1\}$ and $m_{\partial}$ be the point measure $\delta_1$ at $1$, so that  the boundary space is $L^1_{\partial}=\{f_{\partial}\colon \{1\}\to \mathbb{R}:f_{\partial}(1)\in \mathbb{R}\} $ and can be identified  with $\mathbb{R}$, by writing $f_{\partial}=f_{\partial}(1)$. Let  the boundary operators $\Psi_0$ and $\Psi$ be defined by
\[
\Psi_0f=f(1)\quad\text{and} \quad\Psi f =\int_{[0,1]}f(x)\mu(dx), \quad f\in W^{1,1}[0,1],
\]
where $\mu$ is a finite Borel measure.
Note that for each $\lambda>0$ and $f\in \Ker(\lambda I-A)$ we have $f'=\lambda f$. Thus $f'$ is a continuous function. Consequently, for each $f_{\partial}\in L^1_{\partial}$ and $\lambda>0$, the solution $f=\Psi(\lambda)f_\partial$ of equation $f'=\lambda f$ satisfying $\Psi_0(\lambda)f=f_{\partial}$ is of the form
\[
\Psi(\lambda)f_\partial(x)=e^{\lambda (x-1)}f_\partial,\quad x\in [0,1].
\]
Hence condition \eqref{S1} holds true.
We have
\[
\int_{[0,1]} Af(x)dx=f(1)-f(0),\quad f\in W^{1,1}[0,1],
\]
and the restriction $A_0$ of the operator $A$ to
\[
\mathcal{D}(A_0)=\{f\in W^{1,1}[0,1]: f(1)=0\}
\]
is the generator of a positive semigroup. Thus conditions \eqref{S2} and \eqref{S4} hold true.
If there exists $\lambda>0$ such that
\begin{equation}\label{e:mu}
\int_{[0,1]} e^{\lambda (x-1)}\mu(dx)<1,
\end{equation}
then  condition \eqref{S3} holds true and the operator $A_\Psi\subseteq \frac{d}{dx}$ with domain
\[
\mathcal{D}(A_{\Psi})=\{f\in W^{1,1}[0,1]: f(1)=\int_{[0,1]}f(x)\mu(dx)\}
\]
is the generator of a positive semigroup, by Theorem~\ref{l:Apsi}.
Now suppose that $\mu$ is a~probability measure, so that $\mu([0,1])=1$. Then
\[
\int_{[0,1]} e^{\lambda (x-1)}\mu(dx)\le 1
\]
for all $\lambda>0$. Thus if \eqref{e:mu} does not hold for any $\lambda>0$ then $e^{\lambda (x-1)}=1$ for all $\lambda>0$ and $\mu$-almost every $x\in [0,1]$ implying that $\mu\{x\in [0,1]: x=1\}=1$.  Consequently, if $\mu$ is a probability measure such that $\mu\neq \delta_1$ then $(A_\Psi,\mathcal{D}(A_\Psi))$  is the generator of a positive semigroup.
\end{example}

It should  be noted that in \cite[Theorem 4.6]{rudnickityran17}  the assumption that the domain $\mathcal{D}(A_{\Psi})$ of the operator $A_{\Psi}$ is dense is  missing. Making use of Theorem~\ref{l:Apsi}, we get the following result.

\begin{theorem}
Assume conditions \eqref{S1}--\eqref{S4}. If $B$ is a bounded positive operator such that
\[
\int_{E}(A_{\Psi}f(x)+Bf(x))m(dx)\le 0 \quad \text{for all nonnegative }f\in \mathcal{D}(A_{\Psi}),
\]
then $(A_{\Psi}+B,\mathcal{D}(A_{\Psi}))$ is the generator of a substochastic semigroup.
\end{theorem}

We conclude this section with a result concerning the existence of steady states of the positive semigroup from Theorem~\ref{l:Apsi}. Note that given any $\lambda,\mu\in\rho(A_0)$ we have
$
\Psi(\lambda)=\Psi(\mu)+(\mu-\lambda)R(\lambda,A_0)\Psi(\mu)
$, see~\cite[Lemma 1.3]{greiner}.
Thus $\Psi(\lambda)\ge \Psi(\mu)$ for $\lambda\le \mu$. Consequently, for each nonnegative $f_{\partial}\in L^1_{\partial}$  the pointwise limit 
\begin{equation}\label{d:P0}
\Psi(0)f_{\partial}=\lim_{\lambda \to 0^+}\Psi(\lambda)f_{\partial}
\end{equation}
exists and $\Psi(0)f_{\partial}$ is nonnegative.

\begin{theorem}\label{t:mainex} Assume conditions \eqref{S1}--\eqref{S4}. Let $\Psi(0)$ be as in \eqref{d:P0}.
If a nonnegative $f_{\partial}\in L^1_{\partial}$ satisfies $\Psi(0)f_{\partial}\in L^1$  and $f_{\partial}=\Psi\Psi(0)f_{\partial}$,
then $\Psi(0)f_{\partial}\in \mathcal{D}(A_\Psi)$ and $A_\Psi\Psi(0)f_{\partial}=0$.
Conversely, if $A_\Psi f=0$ for a nonnegative $f\in \mathcal{D}(A_\Psi)$ then $f_{\partial}=\Psi f$ satisfies $\Psi\Psi(\lambda)f_{\partial}\le f_{\partial}$ for all $\lambda >\max\{0,\omega\}$, where $\omega$ is as in \eqref{S3}.
\end{theorem}
\begin{proof}
It follows from condition \eqref{S1} that  $\Psi(\lambda)f_{\partial}\in\mathcal{D}$,  $\Psi_0\Psi(\lambda)f_{\partial}=f_{\partial}$,  and $A\Psi(\lambda)f_{\partial}=\lambda f_{\partial}$ for all $\lambda>0$. We have $\Psi(\lambda)f_{\partial}\to \Psi(0)f_{\partial}$ in $L^1$,  as $\lambda\to 0$. Thus $A\Psi(\lambda)f_{\partial}\to 0$  in $L^1$, as $\lambda\to0$. Recall from the proof of Theorem~\ref{l:Apsi} that the operator $(\mathcal{A}+\mathcal{B})(f,0)=(Af,\Psi f-\Psi_0 f)$, $f\in \mathcal{D}$, is a closed operator in the space $L^1\times L^1_{\partial}$.  The operators $\Psi$ and $\Psi_0$ are positive and we have $\Psi\Psi(\lambda)f_{\partial}\to \Psi\Psi(0)f_{\partial}=f_{\partial}=\Psi_0\Psi(0)f_{\partial}$.  Thus, $(\mathcal{A}+\mathcal{B})(\Psi(\lambda)f_{\partial},0)\to (0,0)$ as $\lambda \to 0$.
This implies that $\Psi(0)f_{\partial}\in \mathcal{D}(A_{\Psi})$ and $A_{\Psi}\Psi(0)f_{\partial}=0$.

For the converse, suppose  that $f\in \mathcal{D}(A_\Psi)$ and $A_{\Psi}f=0$. We have $R(\lambda,A_\Psi)(\lambda f-A_\Psi f)=0$. Thus $\lambda R(\lambda,A_\Psi)f=f$ and  $\Psi f=\Psi R(\lambda,A_\Psi)(\lambda f)=\Psi R(\lambda,A_0)(\lambda f)+\Psi \Psi(\lambda) (I_{\partial}-\Psi\Psi(\lambda))^{-1}\Psi R(\lambda,A_0)(\lambda f)$, by \eqref{eq:resbou}. Since $$\Psi R(\lambda,A_0)(\lambda f)=(I_{\partial}-\Psi\Psi(\lambda))(I_{\partial}-\Psi\Psi(\lambda))^{-1}\Psi R(\lambda,A_0)(\lambda f),$$ we conclude that $\Psi f=(I_{\partial}-\Psi\Psi(\lambda))^{-1}\Psi R(\lambda,A_0)(\lambda f)$.
This implies that $(I_{\partial}-\Psi\Psi(\lambda))\Psi f=\Psi R(\lambda,A_0)(\lambda f)\ge 0$ for $\lambda >\max\{0,\omega\}$ and completes the proof.
\end{proof}

\section{A model of a two phase cell cycle in a single cell line}\label{s:cell}

The cell cycle is the period from cell birth to its division into daughter cells. It  contains four major phases: $G_1$ phase (cell growth before DNA replicates), $S$ phase (DNA synthesis and replication), $G_2$ phase (post DNA replication growth period), and    $M$ (mitotic) phase (period of cell division).  The Smith--Martin model \cite{SM73} divides the cell cycle into two phases: $A$ and $B$. The $A$ phase corresponds to all or part of $G_1$ phase of the cell cycle and has a variable duration, while the $B$ phase covers  the rest of the cell cycle.
The cell enters the phase $A$ after birth and waits for some random time $T_A$ until a critical event occurs that is necessary for
cell division. Then the cell enters the phase $B$  which lasts for a finite fixed time $T_B$.
At the end of the $B$-phase the cell splits into two daughter cells.
We assume that individual states of the cell are characterized by age $a\ge 0$ in each phase and by size $x> 0$, which can be volume, mass, DNA content or any quantity conserved trough division.
We assume that individual cells of size $x$  increase their size   over time in the same way, with growth rate $g(x)$ so that $dx/dt=g(x)$, and all cells age over time with unitary velocity so that  $da/dt=1$.
We assume that the probability that a cell is still being in the phase $A$ at age $a$  is equal to $\Ha(a)$, so the rate of exit from the phase $A$ at age $a$ is $\rho(a)$ given by
\begin{equation}\label{d:H}
\rho(a)=-\frac{\Ha'(a)}{\Ha(a)}, \quad  \Ha(a)=\int_a^\infty \psi(r)dr
,
\end{equation}
where $\psi$ is a probability density function defined on $[0,\infty)$, describing the distribution of the time $T_A$, the duration of the phase $A$.
We make the following assumptions:
\begin{enumerate}[(I)]
\item\label{a:Asb} The function $\psi$ in \eqref{d:H} is a probability density function so that  $\psi  [0,\infty)\to [0,\infty)$ is Borel measurable and the function $\Ha$ in \eqref{d:H} satisfies: $\Ha(0)=1$, $\Ha(\infty)=0$.
  \item\label{a:Asa}
The growth rate function $g\colon (0,\infty)\to (0,\infty)$ is globally Lipschitz continuous and  $g(x)>0$ for $x>0$.
\end{enumerate}
The Smith and Martin hypothesis (\cite{SM73}) states that $\psi$ is exponentially distributed with parameter $p>0$, so that $\rho(a)=p$ for all $a>0$.
However, this does not agree with experimental data, see e.g. \cite{golubev2016applications,yates2017multi} for recent results. The generation time of a cell, i.e. the time from birth to division, can be written as $T=T_{A}+T_B$. Thus the distribution of the generation time has a probability density of the form
\[
\psi_T(t)=\left\{
            \begin{array}{ll}
              0, & t<T_B \\
              \psi(t-T_B), & t\ge T_B.
            \end{array}
          \right.
\]
Cell generation times can have lognormal or bimodal distribution (see \cite{sherer2008identification}), exponentially modified Gaussian (\cite{golubev2010exponentially}), or tempered stable distributions (\cite{palmer08}).

To describe the growth of cells we define
\begin{equation}\label{d:Q}
\sQ(x):=\int_{\bar{x}}^x\frac{1}{g(r)}dr ,\quad x>0,
\end{equation}
where $\bar{x}> 0$ or $\bar{x}=0$, if the integral is finite. The value $\sQ(x)$ has a simple biological interpretation. If $\bar{x}$ is the size of a cell, then $\sQ(x)$ is the time it takes the cell to reach the size $x$. It follows from assumption \eqref{a:Asa} that the function  $\sQ$ is strictly increasing and continuous. We denote by $\sQ^{-1}$ the inverse of $\sQ$. Define
\begin{equation}\label{d:pit}
\pi_tx_0=\sQ^{-1}(\sQ(x_0)+t)
\end{equation}
for $t\ge 0$ and $x_0>0$.
Then $\pi_tx_0$ satisfies the initial value problem $$x'(t)=g(x(t)), \quad x(0)=x_0>0.$$
If $\sQ(0)=-\infty$ then $\sQ^{-1}$ is defined on $\mathbb{R}$. Hence, formula \eqref{d:pit} extends to all $t\in \mathbb{R}$ and  $x_0>0$. We also set $\pi_t0=0$ for $t>0$ in this case.
If $\sQ(0)=0$ then $\sQ^{-1}$ is defined only on $(0,\infty)$ and we set $\pi_t0=\sQ^{-1}(t)$ for $t>0$.  We can extend formula \eqref{d:pit} to all negative $t$ satisfying $\sQ(x_0)+t>0$; otherwise we set $\pi_{t}x_0=0$.
Note that at time $t=T$, the generation time, a ``mother cell'' of size $\pi_{T}x_0$ divides into two daughter cells of equal size $\frac12\pi_{T}x_0$.

In the probabilistic model of \cite{hannsgen85a,tyson85,al-mcm-jt-92,tyrcha01}  a sequence of consecutive descendants of a single cell was studied. Let $f$ be the probability density function of the size distribution at birth at time $t_0$ of mother cells and let $t_1>t_0$ be a random time of birth of daughter
cells. Then  the probability density function of the size distribution of daughter cells is given by (\cite{al-mcm-jt-92,tyrcha01})
\begin{equation}\label{op:cc1} P f(x)=-\int _{0}^{\lambda (x)} \dfrac
{\partial}{\partial x} \Ha(\sQ (\lambda(x))-\sQ (r))
    f(r)\, dr,
\end{equation}
where
\begin{equation*}
\lambda(x)=\max\{\pi_{-T_B}(2x), 0\}=\max\{\sQ^{-1}(\sQ(2x)-T_B),0\}.
\end{equation*}
The iterates $P^2f, P^3f,\ldots$ denote densities of the size distribution of consecutive descendants born at random times $t_2, t_3,\ldots.$
The operator $P$ defined by \eqref{op:cc1} is  a positive contraction on $L^1(0,\infty)$, the space of Borel measurable functions defined on $(0,\infty)$ and integrable with respect to the Lebesgue measure.
Here we extend the probabilistic model to a continuous time situation by examining what happens at all times $t$ and not only at $t_0,t_1,t_2,\ldots$.

We denote by $p_1(t,a,x)$ and $p_2(t,a,x)$ the densities of the age and size distribution of  cell in the $A$-phase and in the $B$-phase at time $t$, age $a$, and size $x$, respectively. Neglecting cell deaths the equations can be written as
\begin{equation}\label{a:eq}
\begin{split}
\frac{\partial p_1(t,a,x)}{\partial t}+\frac{\partial p_1(t,a,x)}{\partial
a}+\frac{\partial (g(x)p_1(t,a,x))}{\partial
x}&=-\rho(a)p_1(t,a,x),\\
\frac{\partial p_2(t,a,x)}{\partial t}+\frac{\partial p_2(t,a,x)}{\partial
a}+\frac{\partial (g(x)p_2(t,a,x))}{\partial x}&=0,
\end{split}
\end{equation}
with boundary  and initial conditions
\begin{align}\label{c:bc}
p_1(t,0,x)&=2p_2(t,T_B,2x),\quad x>0, t>0,\\
\label{c:bc2}p_2(t,0,x)&=\int_0^\infty \rho(a)p_1(t,a,x)da, \quad x>0, t>0,\\
p_1(0,a,x)&=f_1(a,x),\quad p_2(0,a,x)=f_2(a,x).\label{c:ic}
\end{align}
In this model,  cells in the A-phase enter the B-phase at rate $\rho$. This is taken into account by the boundary condition \eqref{c:bc2}.  All cells  stay in the B-phase until they reach the age $T_B$. Then they divide their size into half \eqref{c:bc}.
The model is complemented with initial conditions \eqref{c:ic}.
The model we propose is different as compared to mass/maturity structured models \cite{diekmann84,hannsgen85,TH-cc,pichorrudnicki2018} where a cell leaves the phase $A$  with intensity being dependent on maturity, not age.  In the case of $T_B=0$ there is only one phase present; a maturity structured model being a continuous time extension of~\cite{LM-cc} is studied  in \cite{mackeytyran08}, while age and volume/maturity structured population models of growth and division were studied extensively since the seminal work of \cite{bell1967cell} and  \cite{rubinow68,lebowitzrubinow74}. We refer the reader to \cite{metz86diekmann} for  historical remarks concerning modeling of age structured populations and to \cite{sherer2008identification,webb} for recent reviews.

We look for positive solutions of \eqref{a:eq}--\eqref{c:ic} in
the space $L^1=L^1(E,\ES,m)$ with 
 $E=E_1\times \{1\}\cup E_2\times \{2\}$, where
\[
E_1=\{(a,x)\in (0,\infty)\times (0,\infty): x>\pi_a 0\}
\]
and
\[
E_2=\{ (a,x)\in (0,T_B)\times (0,\infty): x>\pi_a 0\},
\]
$m$ is the product of the two-dimensional Lebesgue measure and the counting measure on $\{1,2\}$, and $\ES$ is the $\sigma$-algebra of all Borel subsets of $E$. We identify $L^1=L^1(E,\ES,m)$ with the product of the spaces $L^1(E_1)$ and $L^1(E_2)$ of functions defined on the sets $E_1$ and $E_2$, respectively, and being integrable with respect to the two-dimensional Lebesgue measure.
We say that the operator $P$ has a steady state in $L^1(0,\infty)$ if there exists a probability density function $f$ such that $Pf=f$. Similarly, a semigroup $\{S(t)\}_{t\ge 0}$ has a steady state in $L^1$ if there exists a nonnegative $f\in L^1$ such that $S(t)f=f$ for all $t>0$ and $\|f\|_1=1$ where $\|\cdot\|_1$ is the norm in $L^1$.

\begin{theorem}\label{t:main}
Assume conditions \eqref{a:Asb} and \eqref{a:Asa}.
There exists a unique positive semigroup $\{S(t)\}_{t\ge 0}$ on $L^1$ which provides solutions of \eqref{a:eq}--\eqref{c:ic} and $\{S(t)\}_{t\ge 0}$ is stochastic.  If $\Ha\in L^1(0,\infty)$ then the semigroup $\{S(t)\}_{t\ge 0}$ has a steady state in $L^1$ if and only if the operator $P$ in \eqref{op:cc1} has a steady state in $L^1(0,\infty)$.
\end{theorem}

We give the proof of Theorem~\ref{t:main} in the next section.
Theorem~\ref{t:main} combined with \cite{baronlasota93} implies the following sufficient conditions for the existence of steady states of~\eqref{a:eq}--\eqref{c:ic}.

\begin{corollary}\label{c:barlas}
Assume conditions \eqref{a:Asb} and \eqref{a:Asa}. Suppose that $\Ha\in L^1(0,\infty)$ and that $|\sQ(0)|<\infty$.
If
\begin{equation}\label{e:ETAfinite}
\Ex(T_A):=
\int_0^\infty \Ha(a)da<\liminf_{x\to \infty} (\sQ(\lambda(x))-\sQ(x))
\end{equation}
then \eqref{a:eq}--\eqref{c:ic} has a steady state and it is unique if, additionally, $\psi(a)>0$ for all sufficiently large $a$. Conversely, if there is $x_0\ge 0$ such that $\Ha(Q(\lambda(x_0)))>0$ and
$\Ex(T_A)>\sup_{x\ge x_0} (\sQ(\lambda(x))-\sQ(x))
$, then \eqref{a:eq}--\eqref{c:ic} has no steady states.
\end{corollary}

If the cell growth is exponential so that we have $g(x)=kx$ for all $x>0$, where $k$ is a positive constant,  then  it is known \cite{hannsgen85a,tyson85,tyrcha01} that the operator $P$ has no steady state. We now consider a linear cell growth and assume that $g(x)=k$ for all $x>0$. We see that $\sQ(x)=x/k$, the operator  $P$  is of the form (see \cite{tyson85} or the last section)
\[
Pf(x)=\frac{2}{k}\int_0^{2x-kT_B}\psi((2x-kT_B-r)/k)f(r)dr \mathbf{1}_{(0,\infty)}(2x-kT_B), \quad x>0,
\]
and condition \eqref{e:ETAfinite} holds if and only if $\Ex(T_A)<\infty$.  Combining Corollary~\ref{c:barlas} with Theorem~\ref{t:main} implies the following.
\begin{corollary}
Assume that $g(x)=k$ for $x>0$ and  that $\psi(a)>0$ for all sufficiently large $a>0$. If  $\Ex(T_A)<\infty$ then the semigroup $\{S(t)\}_{t\ge 0}$ has a unique steady state.
\end{corollary}

\section{Proof of Theorem~\ref{t:main}}

We will show that Theorem~\ref{t:main} can be deduced from Theorems~\ref{l:Apsi} and~\ref{t:mainex}. To this end, we introduce some notation. Let us define
\begin{equation*}
\pi(t,a_0,x_0)=(a_0+t,\pi_tx_0),\quad a_0 ,x_0\ge 0, t\in \mathbb{R},
\end{equation*}
where $\pi_t$ is given by \eqref{d:pit}.
Then $t\mapsto \pi(t,a_0,x_0)$ solves the system of equations
 $a'(t)=1$ and $x'(t)=g(x(t))$ with initial condition $a(0)=a_0$ and $x(0)=x_0$. Recall that  $E_1$ is an open set.  For any $x_0,a_0\in E_1$ we define
\[
\nlife(a_0,x_0)=\inf\{s>0:\pi(-s,a_0,x_0)\not\in \overline{E}_1\}
\]
and the incoming part of the boundary $\partial E_1$
\[
\Gamma_1^{-}=\{z\in \partial E_1: z=\pi(-\nlife(y),y) \text{ for some } y\in E_1 \text{ with }\nlife(y)<\infty\}. \]
Observe that $\nlife(a_0,x_0)=a_0$ for all $(a_0,x_0)\in E_1$ and that $\Gamma_1^{-}=\{0\}\times (0,\infty)$. We consider on  $\Gamma_1^{-}$ the Borel measure $m_1^{-}$ being the product of the point measure $\delta_{0}$ at $0$  and the Lebesgue measure on $(0,\infty)$.
We define the operator $\mathrm{T}_{\max}$ on $L^1(E_1)$ by (\cite{arlottibanasiaklods07}) \[
\mathrm{T}_{\max}f(a,x)=-\frac{\partial (f(a,x))}{\partial
a}-\frac{\partial (g(x)f(a,x))}{\partial x}
\]
with domain
\[
\mathcal{D}(\mathrm{T}_{\max})=\{f\in L^1(E_1): \mathrm{T}_{\max}f\in L^1(E_1)\},
\]
where the differentiation is understood in the sense of distributions.
Then it follows from \cite{arlottibanasiaklods07} that for $f\in \mathcal{D}(\mathrm{T}_{\max})$ the following limit
\[
\mathrm{B}^{-}f(z)=\lim_{t\to 0}f(\pi(t,z)) 
\]
exists for almost every $z\in \Gamma_1^{-}$ with respect to the measure $m_1^{-}$ on $\Gamma_1^{-}$.
According to \cite[Theorem 4.4]{arlottibanasiaklods07} the operator $T_0=\mathrm{T}_{\max}$ with domain
\[
\mathcal{D}(\mathrm{T}_0)=\{f\in \mathcal{D}(\mathrm{T}_{\max}): \mathrm{B}^{-}f=0
\}
\]
is the generator of a substochastic semigroup on $L^1(E_1)$ given by
\[
U_0(t)f(a,x)=\frac{g(\pi_{-t}x)}{g(x)}f(a-t,\pi_{-t} x)\mathbf{1}_{\{t<\nlife(a,x)\}}(a,x), \quad (a,x)\in E_1, f\in L^1(E_1).
\]
By \cite[Proposition 5.1]{arlottibanasiaklods07},  the operator $(\mathrm{T},\mathcal{D}(\mathrm{T}))$ defined by
\[
\mathrm{T}f=\mathrm{T}_{\max}f-\rho f, \quad f\in \mathcal{D}(\mathrm{T})=\{f\in\mathcal{D}(\mathrm{T}_0):\rho f\in L^1(E_1)\}
\]
is the generator of a substochastic semigroup on $L^1(E_1)$ of the from
\[
U_1(t)f(a,x)=e^{-\int_0^t \rho(a-r)dr}U_0(t)f(a,x),  \quad (a,x)\in E_1, f\in L^1(E_1).
\]

Note that we can identify the space $L^1(E_2)$ with the subspace
\[
Y=\{f\in L^1(E_1):  f(a,x)=0 \text{ for a.e. } (a,x)\in  E_1\setminus E_2\}
\]
of $L^1(E_1)$
and we have $\mathrm{T}_{\max}(\mathcal{D}(\mathrm{T}_{\max})\cap L^1(E_2))\subseteq L^1(E_2)$.
We set
\[
\nlife(a_0,x_0)=\inf\{s>0: \pi(-s,a_0,x_0)\not\in \overline{E_2}\}=a_0,\quad (a_0,x_0)\in E_2,
\]
and
\[
\Gamma_2^{-}=\{z\in \partial E_2: z=\pi(-\nlife(y),y) \text{ for some } y\in E_2 \text{ with }\nlife(y)<\infty\}.
\]
We also define the exit time from the set $E_2$ by
\[
\life(a_0,x_0)=\inf\{s>0:\pi(s,a_0,x_0)\not\in \overline{E}_2\}
\]
and the outgoing part of the boundary $\partial E_2$
\[
\Gamma_2^{+}=\{z\in \partial E_2: z=\pi(\life(y),y) \text{ for some } y\in E_2\}.
\]
We have $\life(a_0,x_0)=T_B-a_0$ and $\Gamma_2^{+}=\{(T_B,x): x>\pi_{T_{B}}0\}$.
We define the Borel measure $m_2^{-}$ on $\Gamma_2^{-}$ as the measure $m_1^{-}$ and the $m_2^{+}$ on $\Gamma_2^{+}$ as the product of the point measure at $T_B$ and the one dimensional Lebesgue measure.
Since  $U_0(t)(L^1(E_2))\subseteq L^1(E_2)$, the part of the operator $(\mathrm{T}_{0},\mathcal{D}(\mathrm{T}_0))$ in $L^1(E_2)$ is the generator of a substochastic semigroup $\{U_2(t)\}_{t\ge 0}$ in $L^1(E_2)$. Moreover, the following pointwise limits exists
\[
\mathrm{B}^{\pm}f(z)=\lim_{t\to 0}f(\pi(\mp t,z))\quad\text{for } f\in\mathcal{D}(\mathrm{T}_{\max})\cap L^1(E_2)
\]
for  almost every $z\in \Gamma_2^{\pm}$ with respect to the Borel measure $m_2^{\pm}$ on $\Gamma_2^{\pm}$.

Let $E_{\partial}= \Gamma_1^{-}\times\{1\}\cup \Gamma_2^{-}\times\{2\}$,  $\E_{\partial}$ be the $\sigma$-algebra of Borel subsets of $E_{\partial}$ and $m_{\partial}$ be the product of the Lebesgue measure on the line $\{0\}\times(0,\infty)$ and the counting measure on $\{1,2\}$. To simplify the notation we identify $L^1_{\partial}=L^1(E_{\partial},\ES_{\partial},m_{\partial})$ with the product space $L^1(0,\infty)\times L^1(0,\infty)$.
We define  operators $A_1$ and $A_2$ by
\begin{align}
A_1f_1&=\mathrm{T}_{\max}f_1-\rho f_1, \quad f_1\in \mathcal{D}_1=\{f_1\in L^1(E_1): \mathrm{T}_{\max}f_1,\rho f_1\in L^1(E_1)\},\\ A_2 f_2&=\mathrm{T}_{\max}f_2,\quad f_2\in \mathcal{D}_2=\{f_2\in L^1(E_2): \mathrm{T}_{\max}f_2\in L^1(E_2)\}.
\end{align}
We set
\[
\mathcal{D}=\{(f_1,f_2)\in \mathcal{D}_1\times \mathcal{D}_2:\mathrm{B}^{-}f_1,\mathrm{B}^{-}f_2\in L^1(0,\infty)\}
\]
and we define the operator $A$ on $\mathcal{D}$ by setting $Af=(A_1f_1,A_2f_2)$ for $f=(f_1,f_2)\in \mathcal{D}$. We take operators $\Psi_0,\Psi\colon \mathcal{D}\to L^1_{\partial}$ of the form
\begin{equation}
\Psi_0 f=(\mathrm{B}^{-}f_1,\mathrm{B}^{-}f_2),\quad f=(f_1,f_2)\in \mathcal{D},
\end{equation}
and
\begin{equation}
\Psi f(x)=\left( 2\mathrm{B}^{+}f_2(T_B,2x )\mathbf{1}_{(\pi_{T_B}(0),\infty)}(2x), \int_0^{\infty} \rho(a) f_1(a,x)\mathbf{1}_{(0,\infty)}(\pi_{-a}x)da\right) \end{equation}
for $f=(f_1,f_2)\in \mathcal{D}$.
We show that the operator $(A_{\Psi},\mathcal{D}(A_{\Psi}))$ is the generator of a positive  semigroup on $L^1$, where $A_\Psi f=Af$ for $f\in \mathcal{D}(A_{\Psi})=\{f\in \mathcal{D}: \Psi_0f=\Psi f\}$.
To this end, we check that assumptions \eqref{S1}--\eqref{S4} of Theorem~\ref{l:Apsi} from Section \ref{s:pre} are satisfied.

We first show that conditions \eqref{S2} and \eqref{S4} hold.
The operator $A$ restricted to $\mathcal{D}(A_0)=\{(f_1,f_2)\in \mathcal{D}_1\times \mathcal{D}_2: \mathrm{B}^{-}f_1=0, \mathrm{B}^{-}f_2=0\}$ is the generator of the  semigroup $\{S_0(t)\}_{t\ge 0}$ given by
\[
S_0(t)f=(U_1(t)f_1,U_2(t)f_2), \quad t\ge 0, f=(f_1,f_2)\in L^1,
\]
since $\{U_1(t)\}_{t\ge 0}$ and $\{U_2(t)\}_{t\ge 0}$ are semigroups on the spaces $L^1(E_1)$ and $L^1(E_2)$ with the corresponding generators. The semigroup $\{S_0(t)\}_{t\ge 0}$  is substochastic.
For all nonnegative $f=(f_1,f_2)\in \mathcal{D}$ we have
\begin{align*}
\int_E Afdm-\int_{E_\partial} \Psi_0 fdm_{\partial}&=
\int_{E_1} A_1f_1(a,x)d a d x+\int_{E_2}A_2f_2(a,x)d a d x
\\
&\quad -\int_{\Gamma_{1}^{-}}\mathrm{B}^{-}f_1(z)m_{1}^{-}(dz)-\int_{\Gamma_{2}^{-}}\mathrm{B}^{-}f_2(z)m_{2}^{-}(dz).
\end{align*}
By \cite[Proposition 4.6]{arlottibanasiaklods07}, this reduces to
\begin{equation}\label{eq:intA}
\int_E Afdm-\int_{E_\partial} \Psi_0 fdm_{\partial}=-\int_{E_1} \rho(a)f_1(a,x)\,d a d x-\int_{\Gamma_2^{+}}\mathrm{B}^{+}f_2(z)m_2^{+}(dz),
\end{equation}
implying that condition \eqref{S4} holds.

For $f=(f_1,f_2)\in \mathcal{D}$ we can rewrite the equation $\lambda f-Af=0$ as
\begin{align*}
\frac{\partial }{\partial
a}\left(e^{\int_0^a \rho (r)dr}f_1(a,x)\right)&=-\frac{\partial }{\partial x}(g(x)f_1(a,x))-\lambda f_1(a,x) , \\ \frac{\partial }{\partial
a}(f_2(a,x))&=-\frac{\partial }{\partial x}(g(x)f_2(a,x))-\lambda f_2(a,x).
\end{align*}
Hence, we see that the right inverse of $\Psi_0$ when restricted to the nullspace of $\lambda I-A$ is given by
\begin{equation}
\label{eq:psil2p}
\Psi(\lambda)f_{\partial}(a,x)\\=\left(e^{-\lambda a-\int_0^a\rho (r)dr}f_{\partial,1}(\pi_{-a}x), e^{-\lambda a}f_{\partial,2}(\pi_{-a}x)\mathbf{1}_{(0,T_B)}(a)\right)\frac{g(\pi_{-a}x)}{g(x)}
\end{equation}
for $(a,x)\in E_1$ and $f_{\partial}=(f_{\partial,1},f_{\partial,2})\in L^1_{\partial}$.
Moreover, if $(f_1,f_2)=\Psi(\lambda)f_{\partial}$ then
\[
\mathrm{B}^{-}f_1(0,x)=\lim_{t\to 0}f_1(t,\pi_tx)=\lim_{t\to 0}e^{-\lambda t-\int_0^t \rho(r)dr}f_{\partial,1}(x)=f_{\partial,1}(x).
\]
Thus $f_1\in \mathcal{D}_1$. Similarly, $f_2\in \mathcal{D}_2$. Hence, condition \eqref{S1} holds.

To check condition \eqref{S3} take $\lambda>0$ and  $f_{\partial}\in L^1_{\partial}$. For  $(f_1,f_2)=\Psi(\lambda)f_{\partial}$  we have
\[
f_2(a,x)=e^{-\lambda a}f_{\partial,2}(\pi_{-a}x)\frac{g(\pi_{-a}x)}{g(x)}\mathbf{1}_{(0,\infty)}(\pi_{-a}x)\mathbf{1}_{(0,T_B)}(a).
\]
This implies that
\[
\begin{split}
\mathrm{B}^{+}f_2(T_B,x)&=\lim_{t\to 0} f_2(T_B-t,\pi_{-t}x)\\
&=\lim_{t\to 0}  e^{-\lambda(T_B-t)}f_{\partial,2}(\pi_{-T_B}x)\frac{g(\pi_{-T_B}x)}{g(\pi_{-t}x)}\mathbf{1}_{(0,\infty)}(\pi_{-T_B}x)\\
&=e^{-\lambda T_B}f_{\partial,2}(\pi_{-T_B}x)\frac{g(\pi_{-T_B}x)}{g(x)}\mathbf{1}_{(0,\infty)}(\pi_{-T_B}x).
\end{split}
\]
Hence,
\[
\Psi \Psi(\lambda)f_\partial(x)=\left( (\Psi \Psi(\lambda)f_\partial)_1(x), (\Psi \Psi(\lambda)f_\partial)_2(x) \right),
\]
where
\[
(\Psi \Psi(\lambda)f_\partial)_1(x)=2e^{-\lambda T_B}f_{\partial,2}(\pi_{-T_B}(2x))\frac{g(\pi_{-T_B}(2x))}{g(2x)}\mathbf{1}_{(0,\infty)}(\pi_{-T_B}(2x))
\]
and, by \eqref{d:H},
\[
(\Psi \Psi(\lambda)f_\partial)_2(x)= \int_0^{\infty}\psi(a)e^{-\lambda a}
f_{\partial,1}(\pi_{-a}x)\frac{g(\pi_{-a}x)}{g(x)}\mathbf{1}_{(0,\infty)}(\pi_{-a}x)da.
\]
For $f_\partial\in L^1_{\partial}$ we obtain
\begin{align*}
\|\Psi\Psi(\lambda)f_\partial\|
&\le e^{-\lambda T_B}\int_{0}^{\infty}|f_{\partial,2}(z)|dz +\int_0^\infty\psi(a)e^{-\lambda a} da \int_0^{\infty}|f_{\partial,1}(y)| dy\\ &\le \max\left\{e^{-\lambda T_B},\int_0^\infty\psi(a)e^{-\lambda a} da\right\}\|f_{\partial}\|,
\end{align*}
showing that $\|\Psi\Psi(\lambda)\|<1$ for all $\lambda>0$.
Consequently, it follows from Theorem~\ref{l:Apsi} that the operator $(A_{\Psi},\mathcal{D}(A_{\Psi}))$ is the generator of a positive semigroup $\{S(t)\}_{t\ge 0}$  on $L^1$. The semigroup $\{S(t)\}_{t\ge 0}$  is    stochastic, since \eqref{eq:APsiz} holds by \eqref{eq:intA}.

Next assume that $\Ha\in L^1(0,\infty)$. By Theorem~\ref{t:mainex}, it remains to
 look for fixed points of the operator $\Psi\Psi(0)$. Here $\Psi(0)$ defined as in \eqref{d:P0} is,  by \eqref{eq:psil2p}, of the form\begin{equation}\label{eq:psil0}
\Psi(0)f_{\partial}(a,x)=\left(e^{-\int_0^a\rho (r)dr}f_{\partial,1}(\pi_{-a}x),f_{\partial,2}(\pi_{-a}x)\mathbf{1}_{[0,T_B)}(a)\right)\frac{g(\pi_{-a}x)}{g(x)}
\end{equation}
for $(a,x)\in E_1$.
Observe that $\Psi(0)f_{\partial}\in L^1$  for $f_{\partial}\in L^1_{\partial}$, since $e^{-\int_0^a\rho (r)dr}=\Ha(a)$, by \eqref{d:H}, and
\[
\|\Psi(0)f_{\partial}\|\le \int_{0}^\infty \Ha(a)da \int_{0}^\infty |f_{\partial,1}(y)|dy+ T_B \int_{0}^\infty |f_{\partial,2}(y)|dy.
\]
We have $\pi_{-T_B}(2x)=\sQ^{-1}(\sQ(2x)-T_B)=\lambda(x)$ if $2x>\pi_{T_B}0$ and
\begin{equation}\label{d:lambdap}
\lambda'(x)=2\frac{g(\lambda(x))}{g(2x)}\mathbf{1}_{(0,\infty)}(\lambda(x)).
\end{equation}
Hence
\[
(\Psi\Psi(0)f_{\partial})_1(x)=f_{\partial,2}(\lambda(x))\lambda'(x)
\]
and
\[
(\Psi\Psi(0)f_{\partial})_2(x)=\int_{0}^{\infty} \rho(a)e^{-\int_0^a\rho (r)dr}f_{\partial,1}(\pi_{-a}x)\frac{g(\pi_{-a}x)}{g(x)}\mathbf{1}_{(0,\infty)}(\pi_{-a}x)da.
\]
If $f_\partial =\Psi\Psi(0)f_{\partial}$ then $f_{\partial,2}(x)=(\Psi\Psi(0)f_{\partial})_2(x)$ and  $f_{\partial,1}$ satisfies
\begin{align*}
f_{\partial,1}(x)&=(\Psi\Psi(0)f_{\partial})_1(x)\\ &=2\int_{0}^{\infty}\psi(a)f_{\partial,1}(\pi_{-a}(\lambda(x)))\frac{g(\pi_{-a}(\lambda(x)))}{g(2x)}\mathbf{1}_{(0,\infty)}(\pi_{-a}(\lambda(x)))da. \end{align*}
By changing the variables $r=\pi_{-a}(\lambda(x))$, we arrive at the equation
\begin{equation}\label{e:fixpart1}
f_{\partial,1}(x)=\frac{2}{g(2x)}\int_{0}^{\lambda(x)}\psi(\sQ(\lambda(x))-\sQ(r))f_{\partial,1}(r)dr, \quad x>0.
\end{equation}
Equivalently,
$
f_{\partial,1} =Pf_{\partial,1}
$
where $P$ is as in \eqref{op:cc1}. Consequently, equation $\Psi\Psi(0)f_{\partial}=f_{\partial}$ has a solution in $L^1_{\partial}$ if and only if the equation $Pf_{\partial,1}=f_{\partial,1}$ has a solution in $L^1(0,\infty)$. Observe also that the operator $\Psi\Psi(0)$ preserves the $L^1_{\partial}$ norm on nonnegative elements. Hence, if $f_{\partial}\in L^1_{\partial}$ is such that $\Psi\Psi(0)f_{\partial}\le f_{\partial}$ then  $\Psi\Psi(0)f_{\partial}= f_{\partial}$.  Thus the assertion follows from Theorem~\ref{t:mainex}.

\section{Final remarks}\label{s:fr}

Our model can be described as a piecewise deterministic Markov process $\{X(t)\}_{t\ge 0}$. We considered three variables $(a,x,i)$, where $i=1$ if a cell is in the phase $A$, $i=2$ if it is in the phase $B$,
the variable $x$ describes the cell size, and $a$ describes the time which elapsed since the cell entered the $i$th phase. Let $t_0=0$. If we observe consecutive descendants of a given cell and the $n$th generation time is denoted by $t_n$, then $t_{n+1}=s_{n}+T_{B}$ where   $s_n$ is the time when the cell from the $n$th generation  enters the phase $B$,  $n\ge 0$. A newborn cell at time $t_n$ is with age $a(t_n)=0$ and with initial size equal to $x(t_n^-)/2$, where $x(t_n^-)$ is the size of its mother cell.   The cell ages with velocity $1$ and its size grows according to the equations $x'(t)=g(x(t))$ for $t\in (t_n,s_n)$. If the cell enters the phase $B$ then its age is reset to $0$ and its size still grows according to $x'(t)=g(x(t))$ for $t\in (s_n,s_n+T_B)$.  We have
\begin{equation}\label{d:jumpsn}
a(s_n)=0,\quad x(s_n)=x(s_n^-), \quad i(s_n)=2,
\end{equation}
and at the end of the second phase the cell divides into two cells, so that we have
\begin{equation}\label{d:jumptn}
a(t_{n+1})=0,\quad x(t_{n+1})=\frac12 x(t_{n+1}^-), \quad i(t_{n+1})=1.
\end{equation}
Thus the process $X(t)=(a(t),x(t),i(t))$ satisfies the  following system of ordinary differential equations
\begin{equation*}
a'(t)=1,\quad x'(t)=g(x(t)),\quad i'(t)=0,
\end{equation*}
between consecutive  times $t_0,s_0,t_1,s_1,\ldots$, called \emph{jump times}.  At jump times the process is given by \eqref{d:jumpsn} and \eqref{d:jumptn}. If the distribution of $X(0)$ has a density $f$ then $X(t)$ has a density $S(t)f$, i.e.,
\[
\Pr(X(t)\in B_i\times\{i\})=\int_{B_i} (S(t)f)_i(a,x)dadx
\]
for any Borel set $B_i\subset E_i$, where $\{S(t)\}_{t\ge 0}$ is the stochastic semigroup from Theorem~\ref{t:main}.

If $f_{\partial,1}$ is the density of the size distribution at time $t_0=0$ and $f_{\partial,2}$ is the density of the distribution of size at time $s_1$, then the distribution of size at time $t_{1}$  is given by
\[
\Pr(x(t_{1})\le x)=\Pr(\pi_{T_B}x(s_{1})\le 2x)=\Pr(x(s_{1})\le \lambda(x))=\int_0^{\lambda(x)}f_{\partial,2}(z)dz
\]
and
\begin{equation}\label{e:fixpart2}
f_{\partial,2}(z)=\int_0^\infty \psi(a)\hat{\pi}_af_{\partial,1}(z)da,
\end{equation}
where
\[
\hat{\pi}_af_{\partial,1}(z)=f_{\partial,1}(\pi_{-a}z)\frac{g(\pi_{-a}z)}{g(z)}
\mathbf{1}_{(0,\infty)}(\pi_{-a}z)
\]
is the density of the size $x(a)$ of the cell at time $a$, if  $x(0)$ has a density $f_{\partial,1}$. Thus the density of the mass $x(t_{1})$ is given by
\[
\frac{d}{dx}\Pr(x(t_{1})\le x)=f_{\partial,2}(\lambda(x))\lambda'(x)=Pf_{\partial,1}(x)
\]
for Lebesgue almost every $x\in (0,\infty)$, where $P$ is as in \eqref{op:cc1}. Now, if the operator $P$ has a steady state $f_{\partial,1}\in L^1(0,\infty)$ so that $f_{\partial,1}$ satisfies \eqref{e:fixpart1} and if $f_{\partial,2}$ is as in \eqref{e:fixpart2}, then $f^*=(f_1^*,f_2^*)$ given by
\begin{equation}
f_1^*(a,x)=e^{-\int_0^a\rho(r)dr}\hat{\pi}_a f_{\partial,1}(x),\quad f_2^*(a,x)=\hat{\pi}_a f_{\partial,2}(x) \mathbf{1}_{(0,T_B)}(a)
\end{equation}
is the steady state for the semigroup $\{S(t)\}_{t\ge 0}$ existing by Theorem~\ref{t:main}. Moreover, it is unique if $P$ has a unique steady state.

\begin{remark}
 It should be noted that in the two-phase cell cycle model in \cite{pichorrudnicki2018} the rate of exit from the phase $A$ depends on $x$, not on $a$, and that there is no such equivalence between the existence of steady states as presented in Theorem~\ref{t:main}.
 Our results remain true if we assume as in \cite{pichorrudnicki2018} that division into unequal parts takes place.
 Methods as in \cite{rudnickityran17,pichorrudnicki2018} can also be used in our model to study asymptotic behaviour of the semigroup $\{S(t)\}_{t\ge 0}$.  For a different approach to study positivity and asymptotic behaviour of solutions of population equations in $L^1$ we refer to~\cite{rhandi98}.
\end{remark}

We conclude this section with an extension of the  age-size dependent model from \cite{bell1967cell} to a model with two phases. Let $p_i(t,a,x)$  be the function representing the distribution
of cells over all individual states $a$ and $x$ at time $t$ in the phase $A$ for  $i=1$ or $B$ for $i=2$, i.e.,
$\int_{a_1}^{a_2}\int_{x_1}^{x_2}p_i(t,a,x)dadx$ is the number of cells
with age between $a_1$ and $a_2$ and size between $x_1$ and $x_2$ at time $t$ in the given phase. Then $p_1$ and $p_2$ satisfy equations \eqref{a:eq}, \eqref{c:bc2}, \eqref{c:ic} while
the boundary condition
 \eqref{c:bc} takes the form
\begin{equation}\label{e:bc}
p_1(t,0,x)=4p_2(t,T_B,2x),\quad x>0, t>0,
\end{equation}
since a mother cell at the moment of division $T_B$ has size $2x$ and gives birth to two daughters of size $x$ entering the phase $A$ at age $0$.

\begin{theorem}
Assume conditions \eqref{a:Asb} and \eqref{a:Asa}.
Then there exists a unique positive semigroup on $L^1$ which provides solutions of \eqref{a:eq}, \eqref{e:bc}, \eqref{c:bc2}, \eqref{c:ic}.
\end{theorem}
This follows from Theorem~\ref{l:Apsi} in the same way as Theorem~\ref{t:main}, where now to check condition \eqref{S3} we note that
\[
\|\Psi\Psi(\lambda)f_\partial\|
\le \max\left\{2e^{-\lambda T_B},\int_0^\infty\psi(a)e^{-\lambda a} da\right\}\|f_{\partial}\|
\]
for all $f_\partial\in L^1_{\partial}$  and $\lambda>0$, implying that $\|\Psi\Psi(\lambda)\|<1$ for all $\lambda>\omega$ with $\omega=\log2/T_B$.

%

\end{document}